\documentclass[12pt]{amsart}
\usepackage[english]{babel}
\usepackage[utf8x]{inputenc} 
\usepackage{amsmath}
\usepackage{cmll}
\usepackage{amssymb}
\usepackage{amsthm}
\usepackage{bbm}
\usepackage{mathrsfs}
\usepackage{indentfirst}
\usepackage{fancyhdr}
\usepackage{newlfont}
\usepackage{setspace}
\usepackage{verbatim}
\usepackage{hyperref} 
\allowdisplaybreaks

\usepackage{geometry}
\renewcommand{\H}{\mathbb{H}}

\newcommand{\G}{\mathbb{G}}
\newcommand{\M}{\mathbb{M}}
\newcommand{\N}{\mathbb{N}}
\newcommand{\K}{\mathbb{K}}
\newcommand{\R}{\mathbb{R}}
\newcommand{\V}{\mathbb{V}}
\newcommand{\W}{\mathbb{W}}
\newcommand{\bu}{\mbox{\bf 1}}

\newtheorem{theo}{Theorem}[section]
\newtheorem{prop}[theo]{Proposition}
\newtheorem{coro}[theo]{Corollary}

\theoremstyle{definition}

\newtheorem{defi}{Definition}

\newtheorem{rema}{Remark}

\begin{document}

\title{A reverse coarea-type inequality in Carnot groups}
\author{Francesca Corni}
\address{Universit\`a di Bologna\\ 
Dip.to di Matematica \\
Piazza di Porta San Donato, 5, 40126\\
Bologna, Italy}
\email{francesca.corni3@unibo.it}



\begin{abstract}
We prove a coarea-type inequality for a continuously Pansu differentiable function acting between two Carnot groups endowed with homogeneous distances. We assume that the level sets of the function are uniformly lower Ahlfors regular and that the Pansu differential is everywhere surjective. 
\end{abstract}

\maketitle

\section{Introduction}
Geometric measure theory in non-Euclidean metric spaces has been relevantly developed during the last decades. One of the first goals in this line of research is the study of Carnot groups, that are connected, simply connected, nilpotent, stratified Lie groups. These are the simplest models of sub-Riemannian manifolds.
One can canonically associate to each Carnot group a family of non-isotropic dilations defined according to the stratification of the Lie algebra of the group.
% It is natural to define, on each Carnot group, a family of non-isotropic dilations according to the stratification of its Lie algebra. 
We study Carnot groups endowed with a distance that is homogeneous with respect to these dilations.
Within the study of these metric spaces, a long-standing open problem is the validity of the coarea formula for Lipschitz maps acting between two Carnot groups. Up to now, for these mappings only a coarea-type inequality is available \cite{Magnanicoareaineq}. Some stronger results have been proved for specific situations. For instance, one can refer to \cite{Franchi96,Magnanicoarea2005, monser} for Lipschitz real-valued maps
acting on a generic Carnot group, to \cite{Artem, areaimpliescoarea, MagSteTre, montivittone} for continuously Pansu differentiable mappings from a Heisenberg group $\H^n$ to $\R^k$ (where, depending on $k$, higher regularity on the Pansu differential may be required) and to \cite{Karma, Magnaniblowup, Magnani2008} for Euclidean regular maps from a Carnot group to $\R^k$.
Moreover, a very general result has recently been proved in \cite{AntoineSebastiano}. The authors consider two Carnot groups $\G$ and $\M$, endowed with homogeneous distances, an open set $\Omega \subset \G$ and a map $f: \Omega \to \M$, with Pansu differential $Df(x)$ continuous on $\Omega$. Then, the coarea formula holds for $f$ if, at every point $x \in \Omega$, either $Df(x)$ is surjective and $\ker(Df(x))$ can be complemented with a homogeneous subgroup (Definition \ref{complementary}) or $Df(x)$ is not surjective. A key step in the proof of the coarea formula \cite[Theorem 1.3]{AntoineSebastiano} is a suitable implicit function theorem (Theorem \ref{IFT}). Fix a value $m \in \M$ and consider a point $x \in f^{-1}(m)$ such that $Df(x)$ is surjective. Assume that there exists a homogeneous subgroup $\V$ complementary to $\ker(Df(x))$ and choose any homogeneous subgroup $\W$ complementary to $\V$. Then there exist an open neighbourhood $\Omega \subset \G$ of $x$, an open set $U \subset \W$ and a map $\phi:U \subset \W \to \V$ such that $\Omega \cap f^{-1}(m)$ is the intrinsic graph of $\phi$ (Definition \ref{intgraph}). It is not clear how to prove the existence of an analogous parametrization if we assume the Pansu differential $Df(x)$ only to be surjective. In this work we bypass this lack and we prove a weaker coarea-type result, that permits, under a further regularity condition, to deal with more general situations. 
More precisely we prove the following result.
\begin{theo}\label{maintheorem2}
Let $(\G,d_1)$, $(\M, d_2)$ be two Carnot groups endowed with homogeneous distances, of metric dimension $Q, \ P$ and topological dimension $q, \ p$, respectively. Let $f \in C^1_{\G}(\G, \M)$ be a function and assume that $Df(x)$ is surjective at every point $x \in \G$. 
Assume that there exist two constants $\tilde{r}, C>0$ such that for $\mathcal{S}^P$-a.e. $m \in \M$ the level set $  f^{-1}(m) $ is $\tilde{r}$-locally %( $\mathcal{S}^{P}_{d_2}$- almost every should be enough)
$C$-lower Ahlfors $(Q-P)$-regular with respect to the measure $\mathcal{S}^{Q-P}$. Let $\Omega$ be a closed bounded subset of $\G$. 
Then there exists a constant $L=L(C, \G ,p )$ such that
$$\int_{\Omega} C_P(Df(x)) d \mathcal{S}^{Q}(x) \leq L \int_{\M} \mathcal{S}^{Q-P}(f^{-1}(m) \cap \Omega)  d \mathcal{S}^P(m).$$
\end{theo}
The factor $C_P(Df(x))$ is the coarea factor of the Pansu differential $Df(x)$ (Definition \ref{coareafactor}) and $\mathcal{S}^{\alpha}$ denotes the $\alpha$-dimensional spherical Hausdorff measure built with respect to the homogeneous distance. Refer to Definition \ref{lowerahlforsregular} for the notion of locally lower Ahlfors regular set. 

It is immediate to extend Theorem \ref{maintheorem2}, to the case when $\Omega$ is a measurable subset of $\G$ (Theorem \ref{theoremmeasurable}).
As an example of its generality, notice that Theorem \ref{maintheorem2} can be applied to any continuously Pansu differentiable functions $f:\H^1 \to \R^2$ satisfying the requirements.

The proof of Theorem \ref{maintheorem2} is inspired to an abstract procedure presented in \cite{quasiPansu}, where it is used to prove a coarea-type inequality for functions from a metric space to a measure space, for packing-type measures. An analogous argument involving suitable packing measures is adapted here to prove Claim 1 of Theorem \ref{maintheorem}.

By applying Theorem \ref{maintheorem2}, we deduce new results about the slicing of measurable functions on the level sets of $f$ (Corollaries \ref{corfunzione} and \ref{corsommabile}).

In Theorem \ref{maintheorem2} the assumption about the uniform local lower Ahlfors regularity of the level sets of the map $f$ can be read also as a substitute of the existence of a suitable splitting of $\G$.
In fact, this condition is automatically verified if one assumes the existence of a $p$-dimensional homogeneous subgroup $\V \subset \G$ complementary to $\ker(Df(x))$ for every point $x \in \G$ (Corollary \ref{conspezz}). We stress that Corollary \ref{conspezz} is just an example of an application of Theorem \ref{maintheorem2}. In fact, as we discussed above, it can be derived also by the coarea formula in \cite[Theorem 1.3]{AntoineSebastiano}.

\section{Preliminary definitions and results}

When we write $a \lesssim b$, we mean that there exists some positive constant $C$ such that $a \leq Cb$. If $C$ depends on some parameter $d$, it will be specified with a subscript. For instance, by $ a \lesssim_d b$ we mean that there exists a constant $C$ depending on $d$ such that $a \leq C b$. Analogous notations are assumed for $\gtrsim$.
\begin{defi}
A \textit{Carnot group} $\mathbb{G}$ is a connected, simply connected, nilpotent Lie group such that its Lie algebra $\mathrm{Lie}(\G)$ is stratified i.e. there exist linear subspaces $V_1, \ V_2, \ \dots, V_k$ such that
$$ \mathrm{Lie}(\G)=V_1 \oplus \dots \oplus V_k$$
and $$[V_1, V_i]=V_{i+1} \ \ \ \ V_k \neq \{ 0 \} \ \ \ \ V_i= \{ 0 \} \ \ \text{if}\ \ i >k, $$ 
where $[V_1,V_i]=\text{span} \{ [X,Y]  : X \in V_1, \ Y \in V_i\}$.

The number $k$ is called the \textit{step} of $\G$.
\end{defi}
The topological dimension of $\G$ is $q= \sum_{i=1}^k \mathrm{dim}(V_i)$. The number $Q= \sum_{i=1}^k (i   \mathrm{dim}(V_i))$ is called the homogeneous dimension of $\G$.

We denote the left translation associated to an element $x \in \G$ by $\tau_{x}: \G \to \G, \ \tau_x(y)=xy.$

We can naturally introduce on $\mathrm{Lie}(\G)$ a family of non-isotropic linear dilations
$$ \delta_t(v)=\sum_{i=1}^k t^i v_i \ \ \mathrm{if} \ \  v= \sum_{i=1}^k v_i  \ \mathrm{with} \ v_i \in V_i.$$
Since $\G$ is simply connected and nilpotent, the exponential map $\mathrm{exp}:\mathrm{Lie}(\G) \to \G$ is a global diffeomorphism, then we can identify $\G$ with $\mathrm{Lie}(\G)$ and any dilation $\delta_t$ can be identified with the function $\mathrm{exp} \circ \delta_t \circ \mathrm{exp}^{-1}: \G \to \G$, and we denote this map again by $\delta_t$.

A Lie subgroup $\W \subseteq \G$ is called \textit{homogeneous} if it is closed with respect to the family of anisotropic dilations, hence if for every $t>0$, $\delta_t(\W)\subseteq \W$.
% A homogeneous subgroup $\W$ is called vertical if $\W=N_{\ell} \oplus V_{\ell+1} \oplus \dots \oplus V_k$ for some $\ell \in \{ 1, \dots, k\}$ and a linear subspace $N_{\ell} \subset V_{\ell}$.

Through the exponential map according to the Baker-Campbell-Hausdorff formula we can move the group product of $\G$ to an isomorphic polynomial group product on $\mathrm{Lie}(\G)$: for $X,Y \in \mathrm{Lie}(\G)$, we call it $BCH(X,Y)$. In particular $\mathrm{Lie}(\G)$ endowed with $BCH(\cdot, \cdot)$ is isomorphic to $\G$ itself (see for instance \cite[Theorem 4.2]{ricci}), so we identify $\G$ and $\mathrm{Lie}(\G)$ as Lie groups.

We fix a basis of $\G$, $(v_1, \dots, v_q)$ and we identify $\G$ with $\R^q$ through the chosen basis as follows
\begin{equation}
\label{identificazioni}
 \varphi: \G \to \R^q, \ \varphi(p)=(x_1, \dots, x_q) \qquad \mathrm{if} \qquad p= \sum_{i=1}^q x_iv_i.
\end{equation}
The product on $\G$ can be moved to a polynomial group product on $\R^q$ (see \cite[Proposition 2.2.22]{BonfiLanco}).
By the identification of $\G$ with $\mathrm{Lie}(\G)$ and $\R^q$, $\G$ can be seen as $\R^q$ endowed at the same time with the structure of Lie group, with a polynomial group product, and the structure of Lie algebra, and hence of linear space. The inverse of an element with respect to the group product is $(x_1, \dots, x_q)^{-1}=(-x_1, \dots, -x_q)$ while the identity element is the null vector of $\R^q$ and we denote it by $0$.

We assume that $\G$ is a Carnot group endowed with a \textit{homogeneous distance} $d$, that is a distance such that $d(zx, zy)=d(x,y)$ for every $x,y,z \in \G$ and $d(\delta_t(x), \delta_t(y))=td(x,y)$ for every $t>0$ and $x,y \in \G$. 
We set $ \| x \|:=d(x,0)$ for every $x \in \G$; the metric closed ball centered at $x$ of radius $r$ is denoted by $B(x,r):=  \{ y \in \G: d(x,y) \leq r \}$  and for every set $S \subset \G$, we call $\mathrm{diam}(S):= \sup \{ d(x,y) : x,y \in S \}$. Notice that $\mathrm{diam}(B(x,r))=2r$ for all $x \in \G$ and $r>0$, for any fixed homogeneous distance $d$. When nothing more is specified, by "ball" we will mean "closed ball". For a ball $B$, we denote the radius of $B$ by $r(B)$.

We fix on $\G$ a scalar product with respect to which $(v_1, \dots, v_q)$ is an orthonormal basis; extending it by left invariance, we obtain a Riemannian metric $g$ on $\G$. The norm arising from the fixed scalar product turns out to be identified through $\varphi$ with the Euclidean metric on $\R^q$. We will denote this norm on $\G$ by $| \cdot |$.
\begin{prop}\cite[Proposition 5.15.1]{BonfiLanco}
\label{normeconfronto}
Let $\G$ be a Carnot group of step $k$ endowed with a homogeneous distance $d$. For every compact subset $K \subset \G$ there exists a constant $C_K$ such that for any $x \in K$
$$ \frac{1}{C_K} |x| \leq \| x \| \leq C_K |x|^{\frac{1}{k}}.$$
\end{prop}
%We recall now the classical Carath\'eodory's construction of measures; for more details please refer to \cite[Section 2.10]{Federer}.
\begin{defi}[Carath\'eodory's construction]
Let $\mathcal{F} \subset \mathcal{P}(\G) $ be a non-empty family of closed subsets of a Carnot group $\G$ equipped with a homogeneous distance $d$. Let $\zeta: \mathcal{F} \to \R^+$ be a function such that $0 \leq \zeta(S)<\infty$ for any $S \in \mathcal{F}$. If $\delta>0$, and $A \subset \G$, we define
\begin{equation}
\label{caratheodory}
\phi_{\delta, \zeta}(A)= \inf \left\{ \sum_{j=0}^{\infty} \zeta(B_j) \ : \ A \subseteq \bigcup_{j=0}^\infty B_j , \ \mathrm{diam}(B_j) \leq 2 \delta, \ B_j \in \mathcal{F} \right\}.
\end{equation}
If $\mathcal{F}$ coincides with the family of closed balls, $\mathcal{F}_b$, with respect to the distance $d$ and $\zeta_S(B(x,r))=r^{\alpha}$ we call 
$$ \mathcal{S}^{\alpha}(A):= \sup_{\delta>0} \phi_{\delta, \zeta_S}(A)$$
the \textit{$\alpha$-spherical Hausdorff measure} of $A$.
\end{defi}
If $\G$ is a Carnot group of topological dimension $q$ and homogeneous dimension $Q$, then $Q$ is the metric dimension of $\G$ with respect to any homogeneous distance $d$. The spherical measures $\mathcal{S}^m$ are invariant by left translation, hence, for any positive $m$, $\mathcal{S}^m(\tau_x(A))=\mathcal{S}^m(A)$ for every $ x \in \G$ and $A \subset \G$. By the uniqueness of the Haar measure, $\mathcal{S}^Q$ coincides up to a constant with the Lebesgue measure $\mathcal{L}^q$ (for more details please refer to \cite[Propositions 2.19, 2.32]{Notesserra}). Moreover, for any positive $m$, $\mathcal{S}^m(\delta_t(A))=t^m \mathcal{S}^m(A)$ for every $t>0$ and $A \subset \G$.
\begin{defi}[Packing]
Let $N$, $\ell$ be two natural numbers, with $\ell \geq 1$. Let $X$ be a metric space.
An \textit{$\ell$-packing} is a countable collection of closed balls $\{B_i \}$ such that the concentric
balls $\ell B_i$ are pairwise disjoint. An $(N, \ell)$-packing is a collection of balls
$\{ B_i \}$ which is the union of at most $N$ $\ell$-packings
\end{defi}
In the previous definition, and from now on, by $\{ B_i \}$ we mean $\{ B_i \}_{i \in \N}$.
\begin{rema}
\label{rem:Pansu}
In a doubling metric space it is not restrictive to assume that once fixed a number $\ell  \geq 1$, there exists a natural number $N$, only depending on $\ell$, such that, for every $\delta$ small enough, there exist $(N, \ell)$-packings made of balls of radius smaller that $\delta$ that cover the whole space. For instance, in \cite[Remark 3.2]{quasiPansu} it is proved that if a metric space $X$ is doubling at small scales, fine coverings of $(N,\ell)$-packings exist, with $N$ depending only on $\ell$.
\end{rema}
\begin{defi}[Packing premeasure]
Let $\ell \geq 1$ and $N$ be natural numbers. Let $\G$ be a Carnot group endowed with a homogeneous distance $d$ and let $\delta>0$, $\alpha>0$; let $E \subset \G$, we introduce
\begin{equation*}
\begin{aligned}
\mathcal{P}_{N,\ell, \delta}^{\alpha}(E) = \sup \Big\{ \sum_{i=1}^{\infty} r(B_i)^{\alpha} :  \ \{ & B_i \} \ (N, \ell)\text{-} \mathrm{packing  \ of  \ }E , \
   E \subseteq \bigcup_{i=1}^{\infty} B_i, \\
    & B_i \ \mathrm{centered \ on \ }E, \ r(B_i) \leq \delta \Big\}
\end{aligned}
\end{equation*}
and define
$$ \mathcal{P}_{N, \ell}^{\alpha}(E):= \inf_{\delta>0}\mathcal{P}_{  N, \ell, \delta}^{\alpha}(E).$$
\end{defi}
We define also the following packing-type premeasure. In particular, in this case we do not require the packings to cover the set,
\begin{equation*}
\begin{aligned}
\tilde{\mathcal{P}}_{N,\ell, \delta}^{\alpha}(E)= \sup \Big\{ \sum_{i=1}^{\infty} r(B_i)^{\alpha} 
 : \ & \{B_i \} \ (N, \ell)\text{-} \mathrm{packing  \ of  \ }E ,\\
 & B_i \ \mathrm{centered \ on \ }E,  r(B_i) \leq \delta \Big\},
\end{aligned}
\end{equation*}
and
$$ \tilde{\mathcal{P}}_{ N, \ell}^{\alpha}(E):= \inf_{\delta>0}\tilde{\mathcal{P}}_{ N, \ell, \delta}^{\alpha}(E).$$
\begin{rema}
\label{rem:confronto}
Let $\G$ be a Carnot group endowed with a homogeneous distance $d$ and let $E \subset \G$ and $\alpha>0$. Let $\ell \geq 1$ and let $N$ be a natural numbers such that there exist fine $(N, \ell)$-packings of $\G$ that cover $\G$ (see Remark \ref{rem:Pansu}). Then
\begin{equation}
\label{confronto}
\mathcal{S}^{\alpha}(E) \leq  \mathcal{P}^{\alpha}_{N, \ell}(E).
\end{equation}
In fact, for every $\delta>0$, any $(N, \ell)$-packing of $E$ that covers $E$ with balls centered on $E$ of radius smaller that $\delta$ is a covering of $E$ of balls of radius smaller than $\delta$ so, surely, for any $\delta >0$
$$ \phi^{\alpha}_{\delta, \zeta_S}(E) \leq \mathcal{P}^{\alpha}_{   N, \ell, \delta}(E)$$ where $\phi_{\delta, \zeta_S}$ is built, as before, on the family of closed balls $\mathcal{F}_b$.
Letting $\delta $ go to zero, we get $\eqref{confronto}$.
\end{rema}
For any $k \in \N$, we denote by $\mathcal{H}^{k}_E$ the Hausdorff measure on $\G $ i.e. the measure obtained by Carath\'eodory's construction assuming that $\mathcal{F}$ is the family of all closed sets and
$$\zeta(B)=\frac{\mathcal{L}^{k}(\{y\in \R^{k}: |y | \le 1\})}{2^{k}} \mathrm{diam}(B)^{k}.$$ We denote the closed Euclidean ball of center $x$ and radius $r>0$ by $B_E(x,r)=\{ x \in \G \ : \ |x| \leq r \}.$ \\

From now on, we consider two Carnot groups endowed with homogeneous distances $(\G,d_1)$, $(\M, d_2)$, of metric dimension $Q$ and $P$ and topological dimension $q$ and $p$, respectively. The Lie algebras of $\G$ and $\M$ are stratified and we identify as above $\mathrm{Lie}(\G)$ with $\G$ and $\mathrm{Lie}(\M)$ with $\M$ so the groups can be seen as direct sum of linear subspaces 
$$ \G = V_1 \oplus \dots \oplus V_k \qquad \M= W_1 \oplus \dots \oplus W_M.$$
We denote by $\delta_t^1$ and $\delta_t^2$ the anisotropic dilations of parameter $t>0$ on $\G$ and $\M$, respectively.

By $B(x,r)$ and $B_{\M}(x,r)$ we denote the closed metric balls (of center $x$ and radius $r$) in $\G$ and $\M$, respectively. 

A map $L:\G \to \M$ is a \textit{h-homomorphism} if it is a group homomorphism such that $L(\delta_t^1(x))=\delta_t^2(L(x))$ for any $x \in \G$ and $t>0$. In this case we say $L \in \mathcal{L}(\G, \M)$. Given two h-homomorphisms $L, T \in \mathcal{L}(\G, \M)$, we define the distance $d_{\mathcal{L}(\G, \M)}(L,T):= \sup_{q \in B(0,1)} d_2(L(q), T(q))$ and we denote by $\| L \|_{\mathcal{L}(\G, \M)}:= d_{\mathcal{L}(\G, \M)}(L, I)$, where $I: \G \to \M$ denotes the map that associates to any point of $\G$ the unit element of $\M$. 

If we identify $\G$ with $\R^q$ and $\M$ with $\R^p$ through two fixed bases $(v_1, \dots, v_q)$ and $(w_1, \dots, w_p)$ as in (\ref{identificazioni}), any h-homomorphism $L$
is in particular a linear map from $\R^q$ to $\R^p$. We endow, as described above, both $\G$ and $\M$ with a scalar product, with respect to which the two fixed bases are respectively orthonormal. Then we can consider the Jacobian of $L$, $|L|=\sqrt{\det (LL^*)}$. Observe that $|L|$ is the Euclidean algebraic Jacobian of $L$ from $\R^q $ to $\R^p$, or, equivalently, it is the Jacobian of $L$ between the two Lie algebras $\G$ and $\M$ with respect to the fixed scalar products. For more details about h-homomorphisms, please refer to \cite[Section 3.1]{Magnani2001}.

%The set $L^{-1}(0)$ is a homogeneous vertical normal subgroup.
An invertible h-homomorphism is called a \textit{h-isomorphism}.

We denote by $\|x\|_1:=d_1(x,0)$ for every $x \in \G$ and by $\| x \|_2:=d_2(x,0)$ for every $x \in \M$.

Let $\Omega$ be an open set in $\G$ and $f : \Omega \to \M$ be a continuous function. Fix a point $x \in \Omega$.
If there exists a h-homomorphism $L:\G \to \M$ that satisfies 
$$
\| L(x^{-1} y)^{-1} f(x)^{-1} f(y) \|_2= o ( \|x^{-1} y \|_1) \qquad \ \  \text{as} \  \ \|x^{-1} y \|_1  \to 0,
$$
$f$ is said \textit{Pansu differentiable} at $x$. If such a map $L$ exists, it is unique and it is called
the {\em Pansu differential} of $f$ at $x$. We denote it by $Df(x)$. This definition has been introduced in \cite{Pansu}.
We say that $f \in C^1_{\G}(\Omega, \M)$ or that $f$ is \textit{continuously Pansu differentiable} on $\Omega$ if the function 
$Df: \Omega \to \mathcal{L}(\G, \M)$ is continuous. 

The norm of the Pansu differential of a continuously Pansu differentiable map is continuous, more precisely the following holds.
\begin{prop}
\label{Remmodcon}
Let $\G$ and $\M$ be two Carnot groups and let $\Omega \subset \G$ be an open set.
If $f \in C^1_{\G}(\Omega, \M)$, the function $ \| Df \|_{\mathcal{L}(\G, \M)}: \Omega \to \R, \ x \to \| Df(x) \|_{\mathcal{L}(\G,\M)}$ is continuous.
\end{prop}
%Any mapping $f : \Omega \subset \G \to \M$ will be represented in the algebra by the map $F : \Omega \to \mathrm{Lie}(\M)$, $F= \mathrm{exp}_2^{-1} \circ f$ and we denote by $F_j:=\pi_j \circ F$, where $\pi_j: \mathrm{Lie}(\M)\to W_j$ are the canonical projections onto the layers of $\mathrm{Lie}(\M)$. If $f$ is Pansu differentiable at $x \in \Omega$, by \cite[Theorem 4.12]{Magnani_2013}, $F_1$ is Pansu differentiable at $x$ and $DF_1(x)= \pi_1 \circ \mathrm{exp}_2^{-1} \circ Df(x)$ ($DF_1: \Omega \to W_1)$.
\begin{defi}\cite[Definition 4.11]{Magnani_2013}
\label{moduluscont}
Let $f : K \to Y$ be a continuous function from a compact
metric space $(K, d_1)$ to a metric space $(Y,d_2)$. Then we define the \textit{modulus of continuity} of $f$ on $K$ as
$$\omega_{K,f}(t)= \max_{\substack{x,y \in K \\ d_1(x,y) \leq t}} d_2(f(x),f(y)).$$
\end{defi}
If we consider an open set $\Omega \subset \G$ and a map $f : \Omega \subset \G \to \M$; for $j=1, \dots, M$, we call $F_j:=\pi_j \circ f$, where $\pi_j: \M \to W_j$ is the orthogonal projection onto the $j$-th layers of $\M$. If $f$ is Pansu differentiable at $x \in \Omega$, by \cite[Theorem 4.12]{Magnani_2013}, the $F_j$ are Pansu differentiable at $x$ for $j=1,\dots, M$ and in particular $DF_1(x)= \pi_1  \circ Df(x)$ (notice that $DF_1: \Omega \to W_1)$.
\begin{rema}
\label{constants}
In the statement of Theorem \ref{uniformPdiff}, and also later in the paper, we will refer to two geometrical constants $c=c(\G,d)$ and $H=H(\G,d)$, that can be associated to any Carnot group $\G$ endowed with a homogeneous distance $d$.
In order to achieve our main result, Theorem \ref{maintheorem2}, the exact value of these constants will not be relevant. Hence, since their definition is very technical, for the sake of accuracy, we refer to \cite[Lemma 4.9, Definition 4.10]{Magnani_2013} for a precise evaluation. Here we just highlight that $c$ and $H$ depend only on the group $\G$ and the distance $d$.
\end{rema}
\begin{theo}{\cite[Theorem 1.2]{Magnani_2013}}
\label{uniformPdiff}
Let $(\G,d_1)$ and $(\M, d_2)$ be two Carnot groups endowed with homogeneous distances. Let $k$ be the step of $\G$ and $\Omega \subset \G$ be an open subset. Let us consider a map $f \in C^1_{\G}(\Omega, \M)$. Let $\Omega_1, \Omega_2 \subset \G$ be two open subsets of $\G$ such that $\Omega_2$ is compactly contained in $\Omega$ and $$ \{ x \in \G \ : \ d(x, \Omega_1) \leq c H \mathrm{diam}(\Omega_1) \} \subset \Omega_2,$$
where $c=c(\G,d_1)$ and $H=H(\G,d_1)$ are the geometric constants of Remark \ref{constants}.
Then there exists a constant $C$, only depending on $\G$, 

$\max_{x \in \Omega_2} \| DF_1(x) \|_{\mathcal{L}(\G, W_1)}$ 
%norma tra spazi vettoriali direi
and on the modulus of continuity $\omega_{\overline{\Omega_2}, DF_1}$ such that
$$ \frac{d_2(f(x)^{-1}  f(y), Df(x)(x^{-1}y))   }{d_1(x,y)} \leq C [\omega_{\overline{\Omega_2},DF_1}(cH d_1(x,y))]^{1/k^2}$$
for every $x, y \in \overline{\Omega_1}$ with $x \neq y$. 
\end{theo}
\begin{rema}
By \cite[Theorem 4.12]{Magnani_2013}, if $f$ is continuously Pansu differentiable, then $x \to DF_1(x)$ is a continuous map from $\Omega$ to $\mathcal{L}(\G, W_1)$, and so by Proposition \ref{Remmodcon}, the modulus of continuity $\omega_{\overline{\Omega_2}, DF_1}(s)$ goes to zero as $s$ goes to zero.
\end{rema}
% da come è definito, siccome dal teorema di magnani modulo fi contiuity va a zero
\begin{defi}[Coarea factor]
\label{coareafactor}
Let $L: \G \to \M$ be a h-homomorphism and let be $Q \geq P$. We call \textit{coarea factor} of $L$, $C_P(L)$, the unique constant such that 
$$ \mathcal{S}^Q(B(0,1)) C_P(L)= \int_{\M} \mathcal{S}^{Q-P}(L^{-1}(\xi) \cap B(0,1) ) d \mathcal{S}^P(\xi).$$
\end{defi}
By \cite[Proposition 1.12]{Magnanicoareaineq}, whose proof relies on the left invariance of the involved spherical Hausdorff measures, on the uniqueness of the Haar measure on Carnot groups and on the Euclidean coarea formula, $C_P(L)$ is well defined, and it is not equal to zero if and only if $L$ is surjective and in this case it can be computed as follows
\begin{equation}
\label{calcolocoarea}
\begin{aligned}
C_P(L)&= \frac{\mathcal{S}^{Q-P}(\ker(L) \cap B(0,1))}{\mathcal{H}_E^{q-p}(\ker(L) \cap B(0,1))} \frac{\mathcal{S}^P(B_{\M}(0,1))}{\mathcal{L}^p(B_{\M}(0,1))}\frac{\mathcal{L}^q(B(0,1))}{\mathcal{S}^Q(B(0,1))}|L|\\
&= Z \frac{\mathcal{S}^{Q-P}(\ker(L) \cap B(0,1))}{\mathcal{H}_E^{q-p}(\ker(L) \cap B(0,1))} |L|,\\
\end{aligned}
\end{equation}
where $Z=\frac{\mathcal{S}^P(B_{\M}(0,1))}{\mathcal{L}^p(B_{\M}(0,1))}\frac{\mathcal{L}^q(B(0,1))}{\mathcal{S}^Q(B(0,1))}$. Observe that $Z$ is a geometrical constant not depending on $L$. 

We now introduce the definition of Federer density.
\begin{defi}
Let $\G$ be a Carnot group endowed with a homogeneous distance $d$.
Let $\mathcal{F}_b$ be the family of closed balls with positive radius in $\G$. Let $\alpha>0$, $x \in \G$ and let $\mu$ be a Borel regular measure on $\G$. We call \emph{spherical $\alpha$-Federer density of $\mu$ at $x $} the real number $$\theta^{\alpha}(\mu,x):= \inf_{\epsilon>0} \sup \left\{ \frac{\mu(B)}{r(B)^{\alpha}} : x \in B \in \mathcal{F}_b, \ \mathrm{diam}(B) < \epsilon \right\}.$$
\end{defi}

The Federer density can be used to represent the abstract way to differentiate any Borel regular measure absolutely continuous with respect to the $\alpha$-spherical one in a metric space satisfying general hypotheses.

\begin{theo}[{\cite[Theorem 7.2]{Magnani2019}}]
\label{abstractdiff}
Let $\alpha>0$ and let $\mu$ be a Borel regular measure on $\G$ such that there exists an open numerable covering of $\G$ whose elements have $\mu$-finite measure. Let $d$ be a homogeneous distance. If $B \subset A \subset \mathbb{G}$ are Borel sets, then $\theta^{\alpha}(\mu, \cdot)$ is a Borel function on $A$. In addition, if $\mathcal{S}^{\alpha}(A) < \infty$ and $\mu \llcorner A$ is absolutely continuous with respect to $\mathcal{S}^{\alpha} \llcorner A$, then we have 
$$ \mu(B)= \int_B \theta^{\alpha}(\mu,x) \ d  \mathcal{S}^{\alpha}(x).$$
\end{theo}
\begin{defi}
\label{lowerahlforsregular}
Let $(X,d,\mu)$ be a metric measure space, consider a subset $E \subset X$ and two positive numbers $\alpha,C>0$. We say that $E$ is \textit{locally $C$-lower Ahlfors $\alpha$-regular with respect to $\mu$} if there is $\tilde{r}>0$ such that for all $x \in E$ and $0<r < \tilde{r}$,
$$ \mu(B(x,r) \cap E) \geq C r^{\alpha}.$$
If we need to stress the value of $\tilde{r}$, we say that $E$ is $\tilde{r}$-locally $C$-lower Ahlfors $\alpha$-regular with respect to $\mu$.
If $\tilde{r}= \infty$, we say that $E$ is $C$-lower Ahlfors $\alpha$-regular with respect to $\mu$.
\end{defi}
In \cite{Magnanicoareaineq}, also relying on a coarea estimate for Lipschitz
maps in arbitrary metric spaces due to Federer \cite[2.10.25]{Federer}, the author proved a coarea-type inequality. We recall it here, adapting it to our context.
\begin{theo}\cite[Theorem 2.6]{Magnanicoareaineq}
\label{coareamagnani}
Let $A \subseteq \G$ be a measurable set and let $f: A \to \M$ be a Lipschitz map, then
\begin{equation}
\label{magineq}
  \int_{\M} \mathcal{S}^{Q-P}(f^{-1}(m) \cap A) d \mathcal{S}^P(m) \leq  \int_{A} C_P(Df(x)) d \mathcal{S}^{Q}(x).
\end{equation}
\end{theo}
\section{Coarea-type Inequality}
We will need the following simple proposition in order to prove the main theorem.
\begin{prop}
\label{areasubgroup}
Let $\G$ be a Carnot group endowed with a homogeneous distance $d$.
Let $\W \subset \G$ be a homogeneous subgroup of topological dimension $n$ and metric dimension $N$. 
Then for every Borel set $B \subset \W$ we have
$$ \mathcal{L}^n (B) = \sup_{ w \in B(0,1) } \mathcal{H}^n_E(B(w,1) \cap \W) \ \mathcal{S}^N (B).$$
\end{prop}
\begin{proof}
%We denote by $\mathcal{B}(x,r)$ the open ball of radius $r$ and center $x$, and by $\delta_t$, for $t>0$, the intrinsic non isotropic dilation associated to $\G$.
Let us consider $\mu_{\W}(A):=\mathcal{H}^n_E \llcorner \W(A)= \mathcal{L}^n(\W \cap A)$ for any set $A \subset \G$. Since both $\mu_{\W}$ and $\mathcal{S}^{N} \llcorner \W$ are Haar measures on $\W$ (for example refer to \cite[Lemma 3.1]{AntoineSebastiano}), they coincide up to a constant, so we can apply Theorem \ref{abstractdiff}, hence we know that $\mu_{\W}(A)= \int_A \theta^{N}(\mu_{\W}, x) \mathcal{S}^N(x)$.
Let us then compute for any $x \in \W$
$$ \theta^N(\mu_{\W}, x)= \inf_{r>0} \sup_{ \substack{ \{ z : x \in B(z,t) \} \\ 0<t<r}} \frac{\mu_{\W}(B(z,t))}{t^N}.$$
Notice that for every $t>0$, $z \in B(x,t)$, 
%by the left invariance of $\mathcal{L}^n= \mathcal{H}^n_E$ on $\W$ (see \cite[Proposition 1.12]{Magnanicoareaineq})
\begin{equation}
\begin{aligned}
\frac{\mu_{\W}(B(z,t))}{t^N}=& \frac{\mathcal{L}^n( z \delta_t( B(0,1)) \cap \W)}{t^N} \\
= & \frac{\mathcal{L}^n(x^{-1}z\delta_t (B(0,1)) \cap x^{-1} \W)}{t^N}\\
=& \mathcal{L}^n(\delta_{1/t}(x^{-1}z) B(0,1) \cap  \W),
\end{aligned}
\end{equation}
hence 
\begin{equation*}
\begin{aligned}
\theta^N(\mu_{\W}, x)= & \inf_{r>0} 
 \sup_{ \{ w : d(w,0) \leq 1 \}} \mathcal{H}_E^n(  B(w,1) \cap  \W) \\
 = & \sup_{ w \in B(0,1) } \mathcal{H}_E^n(  B(w,1) \cap  \W).
\end{aligned}
\end{equation*}
\end{proof}
Before proving our main result, Theorem \ref{maintheorem2}, we make a preliminary observation.
\begin{rema}
It is immediate to observe that any continuously Pansu differentiable function is locally metric Lipschitz, hence by \cite[Theorem 2.1]{areaimpliescoarea} for every measurable set $A \subset \G$ the function $m \to \mathcal{S}^{Q-P}(A \cap f^{-1}(m))$ is $\mathcal{S}^P$-measurable. %In particular this result follows by \cite[Theorem 2.1]{areaimpliescoarea} combined with Monotone Convergence Theorem applied, for every $m \in \M$, to the sequence of characteristic functions $\{\bu_{\Omega_n \cap f^{-1}(m)} \}_{n \in \M}$ that monotonically converges to $\nearrow \bu_{A \cap f^{-1}(m)}$ as $n$ goes to $\infty$, where we have considered an exhaustion of compact sets $\{\Omega_n\}_{n \in \N}$ for $A$: $\Omega_n \nearrow A$ as $n$ goes to $\infty$. Notice that the measurability  of the map $x \to C_P(Df(x))$ follows by \cite[Theorem 2.6]{Magnanicoareaineq}.
\end{rema}
Theorem \ref{maintheorem2} is a direct consequence of the following result.
\begin{theo}
\label{maintheorem}
Let $(\G,d_1)$, $(\M, d_2)$ be two Carnot groups, endowed with homogeneous distances, of metric dimension $Q, \ P$ and topological dimension $q, \ p$, respectively. Let $\Omega'$ be an open subset of $\G$. Let $f \in C^1_{\G}(\Omega', \M)$ be a function and assume that $Df(x)$ is surjective at every $x \in \Omega'$.

Let $\Omega \Subset \Omega'$ be a closed bounded set such that there exists an open set $\Omega''$ and a positive number $s>0$ such that $\Omega''$ is compactly contained in $\Omega'$ and such that, if we set $\Omega_s:= \{ x \in \G : d(x,\Omega)<s \}$ and $R:=c H \mathrm{diam}(\Omega_s),$ 
\begin{equation}
\label{ipotesimagnani}
\Omega^s_R:= \{ x \in \G  : d(x, \Omega_s) \leq R \} \subset \Omega'',
\end{equation}
where $ c=c(\G,d_1)$ and $H=H(\G,d_1)$ are the two geometric constants associated to $(\G, d_1)$ of Remark \ref{constants}. 
Assume that there exist two constants $\tilde{r}, C>0$ such that for $\mathcal{S}^P$-a.e. $m \in \M$ the level set $ f^{-1}(m) $ is $\tilde{r}$-locally %( $\mathcal{S}^{P}_{d_2}$- almost every should be enough)
$C$-lower Ahlfors $(Q-P)$-regular with respect to the measure $\mathcal{S}^{Q-P}$.
Then there exists a constant $L=L(C, \G, p)$ such that
$$\int_{\Omega} C_P(Df(x)) d \mathcal{S}^{Q}(x) \leq L \int_{\M} \mathcal{S}^{Q-P}(f^{-1}(m) \cap \Omega) d \mathcal{S}^P(m).$$
\end{theo}
\begin{proof}
In the proof we will denote by $B(x,r)$ the closed metric ball of $\G$ with respect of $d_1$ with center $x$ and radius $r$.
Let us preliminarily introduce for $\delta>0$, and $E \subset \G$
$$\mathcal{T}_{\delta}(E)= \inf \left\{ \sum_{i=1}^{\infty} \zeta_T(B_i) : B_i  \in \mathcal{F}_b, \ E \subseteq \bigcup_{i=1}^{\infty}  B_i, \ r(B_i) \leq \delta \right\}$$
with $\zeta_T(B(x,r))=  r^{Q-P} \mathcal{S}^P(f(B(x,r))$. Define the Carath\'eodory's measure
$$ \mathcal{T}(E):= \sup_{\delta>0}\mathcal{T}_{\delta}(E).$$
We define for any $\delta>0$ and $E \subset \G$
\begin{equation*}
\begin{aligned}
\mathcal{K}_{ N,\ell, \delta}(E)= \sup \Big\{ \sum_{i=1}^{\infty} \zeta_T(B_i) :  \ \{ &B_i \} \ (N, \ell) \mathrm{\text{-}packing  \ of  \ }E , \ E \subseteq \bigcup_{i=1}^{\infty} B_i,\\
 & B_i \mathrm{ \ centered  \ on \  } E, \   \  r(B_i) \leq \delta \Big\}
\end{aligned}
\end{equation*}
and
$$ \mathcal{K}_{ N, \ell}(E):= \inf_{\delta>0}\mathcal{K}_{  N, \ell, \delta}(E)$$
Fix $N=N(3, \G)$ the minimum natural number such that there exist fine $(N,3)$-packings of $\G$ that cover $\G$ itself. \\

\textbf{Claim 1} The following inequalities hold
$$ \mathcal{T}(\Omega)\leq \mathcal{K}_{ N, 3}(\Omega) \leq \int_{\M} \mathcal{P}_{ N, 1}^{Q-P} (f^{-1}(m) \cap \Omega) d \mathcal{S}^{P}(m).$$

Let us start proving the second inequality of the claim; we follow the scheme of \cite[Proposition 4.2]{quasiPansu}. 
Let $\delta>0$ and let $\{B_i \}$ be a $(N, 3)$-packing of $\Omega$ that covers $\Omega$ with $r(B_i) \leq \delta$ and $B_i$ centered on $\Omega$. Consider any ball $B_i$ and choose $m \in f(B_i)$. Pick $x_i \in B_i \cap f^{-1}(m) \cap \Omega$ and call $B_{i,m}$ the smallest ball centered at $x_i$ that contains $B_i$. Observe that $B_{i,m} \subseteq 3B_i$, so that, for each $m \in \M$, the collection $\{B_{i,m} \}_{\{i: m \in f(B_i) \}}$ is a $(N, 1)$-packing of $f^{-1}(m) \cap \Omega$ consisting of balls of radius less or equal than $ 3 \delta$ centered on $f^{-1}(m) \cap \Omega$, that covers $f^{-1}(m) \cap \Omega$. Then
\begin{equation}
\begin{aligned}
 \sum_i r(B_i)^{Q-P} \mathcal{S}^P(f(B_i)) &= \sum _i ( \int_{\M} \bu_{f(B_i)}(m) d \mathcal{S}^P(m)) (r(B_i))^{Q-P}\\
&=  \int_{\M}  \sum_i \bu_{f(B_i)}(m) (r(B_i))^{Q-P} d \mathcal{S}^P(m) \\
&=  \int_{\M} \sum_{ \{ i : m \in f(B_i) \}}  (r(B_i))^{Q-P} d \mathcal{S}^P(m)  \\
& \leq \int_{\M} \sum_{ \{i : \exists B_{i,m} \}}  (r(B_{i,m}))^{Q-P} d \mathcal{S}^P(m) \\
& \leq \int_{\M} \mathcal{P}^{Q-P}_{  N, 1, 3 \delta} (f^{-1}(m) \cap \Omega) d \mathcal{S}^P(m).\\
\end{aligned}
\end{equation}
Hence, letting $\delta$ go to zero, by Monotone Convergence Theorem
$$\mathcal{K}_{ N, 3}(\Omega) \leq \int_{\M} \mathcal{P}_{ N, 1}^{Q-P} (f^{-1}(m) \cap \Omega) d \mathcal{S}^{P}(m).$$
The first inequality of the claim follows analogously to the usual comparison between packing and spherical measures (see Remark \ref{rem:confronto}).\\

\textbf{Claim 2}
There exists a constant $T=T(C, \G)$ such that for $\mathcal{S}^P$-a.e. $ m \in \M$ the following inequality holds
\begin{equation}
 \mathcal{P}_{ N, 1}^{Q-P} (f^{-1}(m) \cap \Omega) \leq T  \mathcal{S}^{Q-P}(f^{-1}(m) \cap \Omega). 
\end{equation}
Let us first observe that surely for any $\delta>0$ and $m \in \M$ it is true that
$$\mathcal{P}_{ N, 1, \delta}^{Q-P} (f^{-1}(m) \cap \Omega) \leq N \tilde{\mathcal{P}}_{ 1, 1, \delta}^{Q-P}(f^{-1}(m) \cap \Omega). $$
We want to exploit the hypothesis about the Ahlfors regularity of the level sets $f^{-1}(m)$, for $\mathcal{S}^P$-a.e. $m \in \M$, therefore we fix $0 < \delta < \tilde{r}$. For any fixed $m \in \M$ for which the hypothesis of Ahlfors-regularity holds we consider a $(1, 1)$-packing $\{B_i \} $ of balls of radius smaller than $\delta$, centered on $f^{-1}(m) \cap \Omega$ such that $$ \sum_i r(B_i)^{Q-P} \geq \frac{1}{2} \tilde{\mathcal{P}}_{ 1, 1, \delta}^{Q-P}(f^{-1}(m) \cap \Omega).$$
Then by the Ahlfors-regularity of $f^{-1}(m)$ we have
\begin{equation}
\begin{aligned}
\mathcal{P}_{ N, 1, \delta}^{Q-P} (f^{-1}(m) \cap \Omega) &\leq N \tilde{\mathcal{P}}_{ 1, 1, \delta}^{Q-P} (f^{-1}(m) \cap \Omega) \\
& \leq 2N \sum_i r(B_i)^{Q-P} \\
& \leq 2 \ \frac{N}{C} \ \mathcal{S}^{Q-P} (f^{-1}(m) \cap B_i ) \\
& = 2 \ \frac{N}{C} \ \mathcal{S}^{Q-P} (f^{-1}(m) \cap B_i \cap \overline{\Omega_{\delta}}) \\
&\leq 2 \ \frac{N}{C} \ \mathcal{S}^{Q-P}(f^{-1}(m) \cap \overline{\Omega_{\delta}}), \\
\end{aligned}
\end{equation}
where $\Omega_{\delta}:=\{ x \in \G : d(x,\Omega) < \delta \}$.

Now we let $\delta$ go to zero. Since $\Omega$ is closed, for every $m \in \M$, $ f^{-1}(m) \cap \overline{\Omega_{\delta}} \searrow f^{-1}(m) \cap \Omega$ as $\delta \to 0$, hence by Monotone convergence theorem, %for every $m \in \M$ we get %By \cite[Theorem 2.10.25]{Federer}, if we fix some constant $D>0$, $\mathcal{S}^{Q-P}(f^{-1}(m) \cap \overline{\Omega_{D}}) < \infty$ for $\mathcal{S}^P$-a.e. $m \in \M$. For the values of $m \in \M$ for which this is verified, we can apply Lebesgue convergence theorem. 
for $\mathcal{S}^P$-almost all $m \in \M$, we get
$$\mathcal{P}_{  N, 1}^{Q-P} (f^{-1}(m) \cap \Omega) \leq T \mathcal{S}^{Q-P} (f^{-1}(m) \cap \Omega),$$
where $T=\frac{2N}{C}=\frac{2 N(3,\G)}{C}$, hence $T=T(C, \G)$.\\

By combining Claim 1. and Claim 2. we obtain the existence of a constant $T=T(C, \G)$ such that
$$\mathcal{T}(\Omega)\leq \mathcal{K}_{N, 3}(\Omega) \leq T \int_{\M} \mathcal{S}^{Q-P} (f^{-1}(m) \cap \Omega) d \mathcal{S}^P(m).$$
From now on, we denote by $\delta_t^i$ the intrinsic dilations by $t>0$, on $\G$ for $i=1$ and on $\M$ for $i=2$, respectively.\\

\textbf{Claim 3}
If $k$ is the step of $\G$, 
$$ \mathcal{T}(\Omega) \gtrsim_{k,p}  \int_{\Omega} |Df(x)| d \mathcal{S}^{Q}(x).$$

The proof of Claim 3 is composed of two main steps.

First, we can observe that $Df(x)$ is a continuous function on $\Omega'$, so we can consider the following measure on $\Omega'$: for any $A \subset \Omega'$, $\mu(A):=\int_A |Df(x)| d\mathcal{S}^{Q}(x)$. We want to compare $\mu$ with the Carath\'eodory's measure built with coverings composed of closed balls, measured by the function
$$\zeta_R(B(x,r)):= |Df(x)| r^{Q}.$$ We denote this measure by $\mathcal{R}=\sup_{\delta>0} \phi_{ \delta, \zeta_R}$, where in the definition of $\phi_{\delta, \zeta_R}$ we set $\mathcal{F}=\mathcal{F}_b$.

We want to prove that there exists $\bar{r}>0$ such that for every $0 < r \leq \bar{r}$ and for every $x \in \Omega$, $\mu(B(x,r)) \lesssim   \zeta_R(B(x,r))$. In fact, by \cite[2.10.17(1)]{Federer}, this implies that $\mu(A) \lesssim \mathcal{R}(A)$ for any $A \subseteq \Omega$ (and then also for $A= \Omega$).% (lo vedi subito dalla subaddittività della misura e dalla convergenza uniforme del differenziale.) 

Since $\Omega$ is closed and bounded, it is compact.
The function $| Df( \cdot ) | : \overline{\Omega_s} \to \R, \ x \to | Df(x)|$ is a continuous function on a compact set. Let us fix $\epsilon= \min_{x \in \overline{\Omega_s}} |Df(x)|>0$; it is positive since $Df(x)$ is everywhere surjective by hypothesis. Moreover, the map $| Df( \cdot ) |$ is uniformly continuous, then there exists $r'>0$ such that $ | \ |Df(x)|-|Df(y)|  \ | \leq \epsilon$ if $|x-y| \leq r'$, $x,y \in \overline{\Omega_s}$. 

Let us fix $\bar{r} < \min \{ r', s \}$.
Let us fix $x \in \Omega$ and $0<r \leq \bar{r}$ and let us study
\begin{equation}
\begin{aligned}
\frac{\mu(B(x,r))}{|Df(x)| r^Q} = & \frac{1}{r^Q} \int_{B(x,r)} \frac{|Df(y)| }{|Df(x)|}d \mathcal{S}^Q(y)\\ 
 \leq & \frac{1}{r^Q} \int_{B(x,r)} \frac{| \ |Df(y)|-|Df(x)| \ | }{|Df(x)|}d \mathcal{S}^Q(y) \\
&+ \frac{1}{r^Q} \int_{B(x,r)} \frac{|Df(x)| }{|Df(x)|}d \mathcal{S}^Q(y) \\
 \leq & \frac{1}{r^Q} \int_{B(x,r)} \frac{\epsilon}{\min_{x \in \overline{\Omega_s}}|Df(x)|} \mathcal{S}^Q(y)+ \mathcal{S}^Q(B(0,1))\\
 = & 2 \mathcal{S}^Q(B(0,1))=2b,\\
\end{aligned}
\end{equation}
where $0< b= \mathcal{S}^Q(B(0,1))< \infty$, hence for any $x \in \Omega$ and $0<r \leq \bar{r}$, we have 
$$ \frac{\mu(B(x,r))}{|Df(x)| r^Q} \leq 2b$$ and then, as we said above, $\mu(\Omega) \leq 2b  \mathcal{R}(\Omega)$, so $\mu(\Omega) \lesssim \mathcal{R}(\Omega).$\\

In the second part of the proof, we want to compare $\mathcal{T}(\Omega)$ with $\mathcal{R}(\Omega)$, and, in particular, we want to prove that 
\begin{equation}
\label{eqSF} \mathcal{R}(\Omega) \lesssim_{k,p} \mathcal{T}(\Omega).
\end{equation}
As we observed above, the measure $\mathcal{T}$ is defined as a Carath\'eodory's measure defined with coverings of closed balls measured by the function $\zeta_T(B(x,r))= r^{Q-P} \mathcal{S}^P(f(B(x,r))$. The strategy then will rely on the comparison between $ \zeta_T$ and $\zeta_R$.
In particular we fix $h>0$ such that $h<s$ and we want to prove that there exists $\bar{r}>0$ such that for every $0 < r \leq \bar{r}$ for every $x \in \Omega_h:=\{ y \in \G : \mathrm{dist}(y,\Omega)<h \}$, 
\begin{equation}
\label{eqzeta12}
\zeta_T(B(x,r)) \gtrsim_{k,p} \zeta_R(B(x,r)),
\end{equation}
this would give the desired thesis $\eqref{eqSF}$.

Let us first define for any $x \in \Omega_h$ and $r >0$, $$A_{x,r}:= \delta_{\frac{1}{r}}^2(f(x)^{-1} f(B(x,r)) \qquad \mathrm{and} \qquad A_x:= Df(x)(B(0,1)).$$
The proof will be composed of various steps, and it will be useful to give name to the following conditions:
\begin{align}
 &\lim_{r \to 0} \sup_{x \in \Omega_h} \left|  \bu_{A_{x,r}}(m)- \bu_{A_x}(m) \right| =0 \mathrm{ \  for \ any \ } m \in \M;\label{(a)}\\ 
 &\lim_{r \to 0} \sup_{x \in \Omega_h} \left| \mathcal{S}^P(A_{x,r})- \mathcal{S}^P(A_x) \right|.\label{(b)}
\end{align}
Our strategy consists of proving that $\eqref{(a)} \Rightarrow \eqref{(b)} \Rightarrow \eqref{eqzeta12}$, and then we will conclude the proof by proving $\eqref{(a)}$.
Let us start proving that $\eqref{(b)} \Rightarrow \eqref{eqzeta12}$.

Let $x \in \Omega_h$ and denote by $V(x):=(\ker(Df(x))^{\perp}$. For every $r$ small enough
\begin{equation}
\begin{aligned}
\zeta_T(B(x,r))&= r^{Q-P} \mathcal{S}^{P}(B(x,r))\\
&= r^Q \frac{ \mathcal{S}^{P}(f(B(x,r))}{r^P}\\
&=r^Q \frac{ \mathcal{S}^{P}(f(x)^{-1}f(B(x,r))}{r^P}\\
&=r^Q \mathcal{S}^{P}(\delta_{1/r}^2(f(x)^{-1}f(B(x,r))),
\end{aligned}
\end{equation}
hence
$$\frac{\zeta_T(B(x,r))}{\zeta_R(B(x,r))} = \frac{\mathcal{S}^P(\delta_{1/r}^2(f(x)^{-1} f(B(x,r))}{|Df(x)|} .$$
Observe that the map $Df(x)|_{V(x)}:V(x) \to \M$ is injective and surjective and that $|Df(x)|=|Df(x)|_{V(x)}|$, by the choice of $V(x)$.
If $\pi_{V(x)}$ is the orthogonal projection on $V(x)$, for some geometric constant $G$ we have 
$$\mathcal{S}^P(Df(x)(B(0,1)))=G \ \mathcal{L}^p(Df(x)(B(0,1))= G |Df(x)|  \  \mathcal{L}^{p}(\pi_{V(x)}(B(0,1))$$ by the Euclidean area formula (see also \cite[Lemma 9.2]{Magnani2018}, that ensures that one can see any element $y \in \G$ as $y=m_y (\pi_{V(x)}(y))$, with $m_{y} \in \ker(Df(x))$).
Observe that for every $x \in \Omega_h$, $V(x)=(\ker(Df(x))^{\perp}$ is a linear subspace of constant topological dimension $p>0$, since $Df(x)$ is surjective at any point $x \in \Omega_h$. We notice that the factor $\mathcal{L}^p(V(x) \cap B_E(0,1))$ does not depend on $x$. Remember now that by Proposition \ref{normeconfronto} applied to $K=B(0,1)$ there exists a constant $C_{B(0,1)}$ such that
$$ \frac{1}{C_{B(0,1)}} |x| \leq \| x \|_1 \leq C_{B(0,1)} |x|^{\frac{1}{k}},$$
where $k$ is the step of $\G$.
Hence
\begin{equation}
\begin{aligned}
\mathcal{L}^{p}(\pi_{V(x)}(B(0,1)) )&\geq \mathcal{L}^{p}(V(x) \cap B(0,1)) \\
&\geq \mathcal{L}^{p}\left(V(x) \cap B_E\left(0,\frac{1}{(C_{B(0,1)})^k}\right)\right) \\
&= \frac{1}{(C_{B(0,1)})^{kp}} \mathcal{L}^p(V(x) \cap B_E(0,1))\\
&:=D(k,p) >0.\\ 
\end{aligned}
\end{equation}
Hence for every $x \in \Omega_h$, we have
$$\frac{\mathcal{S}^P(A_x)}{|Df(x)|}  \geq  GD(k,p):=D'(k,p)=D' >0.$$

If we now assume $\eqref{(b)}$ to be true, and we fix $\epsilon = \frac{ D'\min_{x \in \overline{\Omega_h}} |Df(x)|}{2}>0$, there exists $0 <\bar{r} \leq s-h$ such that for every $0<r \leq \bar{r}$ and for every $x \in \Omega_h$, 
$$\left| \mathcal{S}^P(A_{x,r})- \mathcal{S}^P(A_x) \right| \leq \sup_{x \in \Omega_h} \left| \mathcal{S}^P(A_{x,r})- \mathcal{S}^P(A_x) \right| \leq \epsilon,$$
so that for every $0<r \leq \bar{r}$ and for every $x \in \Omega_h$, 
$$   \mathcal{S}^P(A_{x,r}) \geq \mathcal{S}^P(A_x) - \epsilon $$
and so for every $x \in \Omega_h$ and $0 < r \leq \bar{r}$
\begin{equation}
\begin{aligned}
 \frac{\zeta_T(B(x,r))}{\zeta_R(B(x,r))} &=  \frac{\mathcal{S}^P(A_{x,r})}{|Df(x)|} \geq \frac{\mathcal{S}^P(A_x)}{|Df(x)|} - \frac{\epsilon}{|Df(x)|}\\ & \geq D'- \frac{\epsilon}{|Df(x)|} \geq D'- \frac{\epsilon}{\min_{x \in \overline{\Omega_h}}|Df(x)|} =\frac{D'}{2}>0
\end{aligned}
\end{equation}
by the choice of $\epsilon$. This gives that $\eqref{(b)} \Rightarrow (\ref{eqzeta12})$.\\

Second point, we prove that $\eqref{(a)} \Rightarrow \eqref{(b)}$. Surely, we know that
\begin{equation}
\label{uptoLebesgue}
\begin{aligned}
&\lim_{r \to 0} \sup_{x \in \Omega_h} \left| \mathcal{S}^P(A_{x,r})- \mathcal{S}^P(A_x) \right|\\ \leq & \lim_{r \to 0} \sup_{x \in \Omega_h} \left| \int_{\M} \bu_{A_{x,r}}(m) - \bu_{A_x}(m) d \mathcal{S}^P(m)\right|\\
\leq & \lim_{r \to 0}  \int_{\M} \sup_{x \in \Omega_h}  \left| \bu_{A_{x,r}}(m) - \bu_{A_x}(m) \right| d \mathcal{S}^P(m).\\
\end{aligned}
\end{equation}
We want now to apply the Lebesgue dominated convergence theorem using $\eqref{(a)}$. In order to do this, we prove that for $r \leq s-h$, for any $m \in \M$
\begin{equation}
\label{equibounded}
 \sup_{x \in \Omega_h}  \left| \bu_{A_{x,r}}(m) - \bu_{A_x}(m) \right| \leq 2 \bu_{B(0,W)}(m) 
\end{equation}
for some constant $W>0$. Notice that $2 \bu_{B(0,W)} \in L^1_{\mathcal{S}^P}(\M).$

First, consider that
$\sup_{x \in \Omega_h}  \left| \bu_{A_{x,r}}(m) - \bu_{A_x}(m) \right| \leq \sup_{x \in \Omega_h}  \left| \bu_{A_{x,r}}(m) \right| + \sup_{x \in \Omega_h} \left| \bu_{A_x}(m) \right|$.

For any $x \in \Omega_h$ and $m \in \M$, if we assume $\bu_{A_x}(m)=1$, it implies that $m = Df(x)(\eta)$ for some $\eta \in B(0,1)$, then $\| m \|_2 = \|Df(x)(\eta) \|_2 \leq \| Df(x) \|_{\mathcal{L}(\G, \M)} \leq \max_{x \in \overline{\Omega_h}} \|Df(x)\|_{\mathcal{L}(\G, \M)} =: \|Df\|_{\overline{\Omega_h}}$.

For any $x \in \Omega_h$, $r \leq s-h$ and $m \in \M$, if $\bu_{A_{x,r}}(m)=1$, $m=\delta_{1/r}^2(f(x)^{-1} f(q_r))$ for some $q_r \in B(x,r) \subseteq \Omega_s$. Hence
\begin{align}
 \|m \|_2 = &\|\delta_{1/r}^2(f(x)^{-1} f(q_r))\|_2 \notag\\
 = &\| Df(x)(\delta_{1/r}^1 (x^{-1}q_r))  \delta_{1/r}^2(Df(x)(x^{-1}q_r))^{-1}f(x)^{-1} f(q_r)) \|_2 \notag \\
 \leq &\| Df(x)(\delta_{1/r}^1 (x^{-1}q_r))  \|_2 + \| \delta_{1/r}^2(Df(x)(x^{-1}q_r))^{-1}f(x)^{-1} f(q_r) \|_2 \\
 \leq &\|Df(x) \|_{\mathcal{L}(\G, \M)} + K(\omega_{\overline{\Omega''}, DF_1}(Hc (s-h)))^{\frac{1}{k^2}}\notag\\
 \leq &\|Df\|_{\overline{\Omega_h}}+ K (\omega_{\overline{\Omega''}, DF_1}(Hc (s-h)))^{\frac{1}{k^2}},\notag
\end{align}
where $\omega_{\overline{\Omega''}, DF_1}$ is the modulus of continuity of $x \to DF_1(x)$ defined in Definition \ref{moduluscont} and $K$ is a constant that plays the role of $C$ of Theorem \ref{uniformPdiff}.

Hence 
\begin{align}
\sup_{x \in \Omega_h}  \left| \bu_{A_{x,r}}(m) - \bu_{A_x}(m) \right| \leq & \sup_{x \in \Omega_h}  \bu_{A_{x,r}}(m) + \sup_{x \in \Omega_h}  \bu_{A_x}(m) \notag\\
\leq & \bu_{B(0,\|Df\|_{\overline{\Omega_h}}+  K(\omega_{\overline{\Omega''}, DF_1}(Hc (s-h)))^{\frac{1}{k^2}})}(m)\notag \\
& +  \bu_{B(0, \|Df\|_{\overline{\Omega_h}}) }(m) \\
 \leq & 2 \bu_{B(0,\|Df\|_{\overline{\Omega_h}}+  K(\omega_{\overline{\Omega''}, DF_1}(Hc (s-h)))^{\frac{1}{k^2}})}(m)\notag
\end{align}
and this implies that (\ref{equibounded}) is true, with $$W=\|Df\|_{\overline{\Omega_h}}+  (\omega_{\overline{\Omega''}, DF_1}(Hc (s-h)))^{\frac{1}{k}}).$$
We can then apply the Lebesgue dominated convergence theorem to (\ref{uptoLebesgue}), and since we have assumed $\eqref{(a)}$ to be true, we obtain $\eqref{(b)}$.\\

It remains to prove $\eqref{(a)}$.

By contradiction, we assume $\eqref{(a)}$ to be false. Then, there exists at least one element $m \in \M$ such that the limit 
$$\lim_{r \to 0}\sup_{x \in \Omega_h} \left|  \bu_{A_{x,r}}(m)- \bu_{A_x}(m) \right| $$
does not exist or
$$\lim_{r \to 0}\sup_{x \in \Omega_h} \left|  \bu_{A_{x,r}}(m)- \bu_{A_x}(m) \right| >0.$$
In both cases, since all the considered elements are positive, there exists at least a positive infinitesimal sequence $r_n$ such that $$\lim_{n \to \infty} \sup_{x \in \Omega_h} \left|  \bu_{A_{x,r_n}}(m)- \bu_{A_x}(m) \right| >0.$$
%(se il limite non esiste, allora hai almeno una successione su cui non esiste, ma esiste il lim sup che non puo essere zero, se no il limite della successione esisterebbe e sarebe zero, e allora puoi estrarre una sottosuccessione che ha limite maggiore di zero, precisamente uguale al limsup, vedi Osservazione 4.4 Lanconelli 1)
This implies that there exists $\tilde{n}>0$ such that for every $n \geq \tilde{n}$, $$\sup_{x \in \Omega_h} \left|  \bu_{A_{x,r_n}}(m)- \bu_{A_x}(m) \right| >0 \qquad \mathrm{and} \qquad r_n \leq s-h.$$ Hence for every $n \geq \tilde{n}$ there exists at least an element $x_n \in \Omega_h \subseteq \overline{\Omega_h}$ such that
$$\left|  \bu_{A_{x_n,r_n}}(m)- \bu_{A_{x_n}}(m) \right| >0$$ and then 
\begin{equation}
\label{absurd2}
\left|  \bu_{A_{x_n,r_n}}(m)- \bu_{A_{x_n}}(m) \right| =1.
\end{equation}
Since $\overline{\Omega_h}$ is a compact set, the sequence $x_n$ converges up to a subsequence to some $\bar{x} \in \overline{\Omega_h}$.\\

Let us first prove that there exists some $\bar{n}$ such that for every $n \geq \bar{n}$
\begin{equation}
\label{ipotesi}
\left|  \bu_{A_{x_n,r_n}}(m)- \bu_{A_{\bar{x}}}(m) \right| =1.
\end{equation}
Let us then assume by contradiction that there exists a subsequence $x_{n_k}$ such that
\begin{equation}
\label{absurd}
\lim_{k \to \infty} \left|  \bu_{A_{x_{n_k},r_{n_k}}}(m)- \bu_{A_{\bar{x}}}(m) \right|=0
\end{equation} then, on this subsequence, we have
\begin{equation}
\label{equazionevaazero}
\begin{aligned}
\left|  \bu_{A_{x_{n_k},r_{n_k}}}(m)- \bu_{A_{x_{n_k}}}(m) \right| \leq & \left|  \bu_{A_{x_{n_k},r_{n_k}}}(m)- \bu_{A_{\bar{x}}}(m) \right|\\
&+ \left|  \bu_{A_{x_{n_k}}}(m)- \bu_{A_{\bar{x}}}(m) \right|.
\end{aligned}
\end{equation}
Let us prove that (\ref{equazionevaazero}) goes to zero as $k \to \infty$ and this would give a contradiction with (\ref{absurd2}). The fact that (\ref{equazionevaazero}) goes to zero follows by the assumption (\ref{absurd}) and by the fact that 
\begin{equation}
\label{daprovare}
 \left|  \bu_{A_{x_{n_k}}}(m)- \bu_{A_{\bar{x}}}(m) \right| \to 0  \ \mathrm{as} \ k \to \infty.\end{equation}
Let us prove (\ref{daprovare}): assume by contradiction that on a subsequence of $ x_{n_k}$, $x_{n_{k_p}}$, for $p$ sufficiently large, 
\begin{equation}
\label{absurd3} \left|  \bu_{A_{x_{n_{k_p}}}}(m)- \bu_{A_{\bar{x}}}(m) \right| =1.
\end{equation}
Let us assume $m \notin A_{\bar{x}}$, then $m \in A_{x_{n_{k_p}}}$ so that $m= Df(x_{n_{k_p}})(\eta_{n_{k_p}})$ for $\eta_{n_{k_p}} \in B(0,1)$, hence we can extract a converging subsequence such that $\eta_{n_{k_{p_q}}} \to \bar{\eta} \in B(0,1)$ as $q \to \infty$ and letting $q$ go to infinity, by the continuity of the differential, we know that $m= Df(\bar{x})(\bar{\eta}) \in A_{\bar{x}}$, but this is not possible. 

So it must be true that $m \in A_{\bar{x}}$. Hence, by the definition of $A_{\bar{x}}$ and the continuity of the Pansu differential, it is true that for some $\eta \in B(0,1)$, $m= Df(\bar{x})(\eta)=\lim_{p \to \infty} Df(x_{n_{k_p}})(\eta)=\lim_{p \to \infty} X_p$ where for every $p$, $X_p:=Df(x_{n_{k_p}})(\eta) \in A_{x_{n_{k_p}}}$. Hence $m \in \limsup_{p \to \infty} A_{x_{n_{k_p}}}$, and then $ \bu_{ \limsup_{p \to \infty} A_{x_{n_{k_p}}}}(m)= \limsup_{p \to \infty} \bu_{A_{x_{n_{k_p}}}}(m)=1$ but at the same time, by (\ref{absurd3}), $m \notin A_{x_{n_{k_p}}}$ for $p$ sufficiently large and this implies that $$\lim_{p \to \infty} \bu_{A_{x_{n_{k_p}}}}(m)=0 = \limsup_{p \to \infty} \bu_{A_{x_{n_{k_p}}}}(m)$$and this gives a contradiction.
Then (\ref{ipotesi}) is proved.\\ Let us continue from (\ref{ipotesi}). We need to prove that (\ref{absurd2}) is not possible. Since $m$ is fixed, there are only two possibilities: 
\begin{align}
m \in A_{\bar{x}}; \label{(a1)}\\
m \notin A_{\bar{x}}. \label{(a2)}
\end{align}
We show that neither $\eqref{(a1)}$ nor $\eqref{(a2)}$ can be true.
Assume that $\eqref{(a1)}$ is true, then $m=Df(\bar{x})(\eta)$ for some $\eta \in B(0,1)$. Then, by (\ref{ipotesi})
$m \notin A_{x_n, r_n}$ for $n \geq \bar{n}$ and so clearly there exists the following limit 
\begin{equation}
\label{ipotesi2}
 \lim_{n \to \infty} \bu_{A_{x_n,r_n}}(m)=0.
\end{equation}
Let us define for any $n$, $q_n:= x_n \delta_{r_n}^1(\eta) \in B(x_n, r_n) \subseteq \Omega_s$.
Consider for any $n \geq \bar{n}$
\begin{equation*}
\begin{aligned}
 \delta_{1/r_n}^2(f(x_n)^{-1} f(q_n)) = & Df(\bar{x})(\eta) \delta_{1/r_n}^2(Df(\bar{x})(x_n^{-1} q_n )^{-1}Df(x_n)(x_n^{-1} q_n ))\\
 &\delta_{1/r_n}^2(Df(x_n)(x_n^{-1} q_n )^{-1} f(x_n)^{-1} f(q_n))
\end{aligned}
\end{equation*}
and observe that by  Theorem \ref{uniformPdiff}, since $x_n, q_n \in \Omega_s$ and (\ref{ipotesimagnani}) holds
$$ \| \delta_{1/r_n}^2(Df(x_n)(x_n^{-1} q_n )^{-1}f(x_n)^{-1} f(q_n)) \|_2 \leq K (\omega_{\overline{\Omega''},DF_1}(cH r_n))^{\frac{1}{k^2}} \to 0 $$ as $n \to \infty$ and 
$$ \| (Df(\bar{x})(\eta))^{-1}Df(x_n)(\eta) \|_2 \leq d_{\mathcal{L}(\G, \M)}(Df(x_n), Df(\bar{x})) \to 0$$ as $n \to \infty$ by the continuity of $Df(x)$.
Hence $ \lim_{n \to \infty} \delta_{1/r_n}^2(f(x_n)^{-1} f(q_n))=m$. This permits to conclude that $m \in \limsup_{n \to \infty} A_{x_n, r_n}$ and so that 
$$\limsup_{n \to \infty} \bu_{A_{x_n,r_n}}(m)=\bu_{\limsup_{n \to \infty} A_{x_n, r_n}}(m)=1.$$
At the same time, (\ref{ipotesi2}) implies that there exists the limit 
$$ \limsup_{n \to \infty} \bu_{A_{x_n,r_n}}(m)= \lim_{n \to \infty} \bu_{A_{x_n,r_n}}(m)=0,$$
so we reach a contradiction.\\

Assume now $\eqref{(a2)}$, then $m \in A_{x_n, r_n}$ for every $n \geq \bar{n}$ and then by (\ref{ipotesi}), $m \notin A_{\bar{x}}$. For every $n \geq \bar{n}$ there exists $q_n \in B(x_n, r_n) \subseteq \Omega_s$ such that 
\begin{equation}
\begin{aligned}
m= &\delta_{1/r_n}^2(f(x_n)^{-1} f(q_n))\\
= & Df(\bar{x})(\delta_{1/r_n}^1 (x_n^{-1}q_n)) Df(\bar{x})(\delta_{1/r_n}^1 (x_n^{-1}q_n))^{-1}\\
&Df(x_n)(\delta_{1/r_n}^1 (x_n^{-1}q_n)) \delta_{1/r_n}^2(Df(x_n)(x_n^{-1} q_n )^{-1}f(x_n)^{-1} f(q_n)) 
\end{aligned}
\end{equation}
 and again by Theorem \ref{uniformPdiff} and by the continuity of $Df$ we obtain that
$$m= \lim_{n \to \infty} \delta_{1/r_n}^2(f(x_n)^{-1} f(q_n))= \lim_{n \to \infty} Df(\bar{x})(\delta_{1/r_n}^1 (x_n^{-1}q_n))$$  that up to a subsequence is equal to $Df(\bar{x})(\eta)$ for some $\eta \in B(0,1)$, so that $m \in A_{\bar{x}}$ that is a contradiction with (\ref{absurd2}).
Hence, finally $\eqref{(a)}$ is proved and this concludes the proof of Claim 3.\\

\textbf{Claim 4} For every $x \in \Omega'$, 
$$C_P(Df(x))\lesssim_{q,p,k} \ |Df(x)|.$$

Since we have assumed that $Df(x)$ is surjective at every $x$, by (\ref{calcolocoarea}), we have
$$C_P(Df(x))= Z \frac{\mathcal{S}^{Q-P}(\ker(Df(x)) \cap B(0,1))}{\mathcal{H}_E^{q-p}(\ker(Df(x)) \cap B(0,1))} |Df(x)|. $$

By Proposition \ref{areasubgroup}, for any $x \in \Omega'$ and any Borel set $B \subset \ker(Df(x))$
$$ \mathcal{S}^{Q-P} (B)= \frac{1}{\sup_{ w \in B(0,1) } \mathcal{H}^{q-p}_E(\ker (Df(x)) \cap B(w,1))} \mathcal{H}^{q-p}_E(B),$$
hence, by taking into account Proposition \ref{normeconfronto}, we get
\begin{equation}
\begin{aligned}
&\frac{\mathcal{S}^{Q-P}(\ker(Df(x)) \cap B(0,1))}{\mathcal{H}_E^{q-p}(\ker(Df(x)) \cap B(0,1))} \\ =& \ \frac{1}{ \sup_{ w \in B(0,1)} \mathcal{H}_E^{q-p}(  B(w,1) \cap  \ker (Df(x)) } \\
\leq & \ \frac{1}{\mathcal{H}^{q-p}_E(B(0,1) \cap \ker (Df(x))}\\
 \leq & \ \frac{1}{\mathcal{H}^{q-p}_E(B_E(0,\frac{1}{(C_{B(0,1)})^{k}})) \cap \ker (Df(x))} \\
=& \ \frac{1}{\mathcal{L}^{q-p}(B_E(0,\frac{1}{(C_{B(0,1)})^{k}}) \cap \ker (Df(x) )}=:D''(q,p,k)>0.
\end{aligned}
\end{equation}
In the last passage we considered that $\ker (Df(x))$ is a linear subspace of constant topological dimension $q-p$. \\
By combining all the claims the proof is finally achieved.
\end{proof}
It is easy to extend Theorem \ref{maintheorem2} to the case in which $\Omega$ is not necessarily compact but it is any measurable set.
\begin{theo}\label{theoremmeasurable}
Let $(\G,d_1)$, $(\M, d_2)$ be two Carnot groups, endowed with homogeneous distances, of metric dimension $Q, \ P$ and topological dimension $q, \ p$, respectively. Let $f \in C^1_{\G}(\G, \M)$ be a function and assume $Df(x)$ to be surjective at any point $x \in \G$. Let $A \subset \G$ be a measurable set. 
Assume that there exist two constants $\tilde{r}, C>0$ such that for $\mathcal{S}^P$-a.e. $m \in \M$ the level set $  f^{-1}(m) $ is $\tilde{r}$-locally %( $\mathcal{S}^{P}_{d_2}$- almost every should be enough)
$C$-lower Ahlfors $(Q-P)$-regular with respect to the measure $\mathcal{S}^{Q-P}$.
Then there exists a constant $L=L(C, \G ,p )$ such that
$$\int_{A} C_P(Df(x)) d \mathcal{S}^{Q}(x) \leq L \int_{\M} \mathcal{S}^{Q-P}(f^{-1}(m) \cap A) \ d \mathcal{S}^P(m).$$
\end{theo}
\begin{proof}
Let us consider an increasing sequence of compact sets in $\Omega_n \subseteq A$ such that $\Omega_n \nearrow A$. Hence by Theorem \ref{maintheorem2}, there exists $L=L(C,\G,p)$ such that for every $n \in \N$
\begin{align*}
\int_{\Omega_n} C_P(Df(x)) d \mathcal{S}^{Q}(x) &\leq L \int_{\M} \mathcal{S}^{Q-P}(f^{-1}(m) \cap \Omega_n) \ d \mathcal{S}^P(m) \\ &\leq L \int_{\M} \mathcal{S}^{Q-P}(f^{-1}(m) \cap A) \ d \mathcal{S}^P(m),
\end{align*}
so if we let $n$ go to $\infty$, by Monotone Convergence Theorem we get the thesis.
\end{proof}
\section{Applications}
\begin{coro}
In the hypotheses of Theorem \ref{theoremmeasurable},
let $u: A \to \R$ be a non-negative measurable function, then there exists a constant $L=L(C, \G ,p )$ such that
$$\int_{A} u(x) C_P(Df(x)) d \mathcal{S}^{Q}(x) \leq L \int_{\M} \int_{f^{-1}(m) \cap A} u(x) d\mathcal{S}^{Q-P}(x)  d \mathcal{S}^P(m).$$
\end{coro}
\begin{proof}
We can write $u= \sum_{k=1}^{\infty} \frac{1}{k} 1_{A_k}$ with $A_k$ measurable sets (see \cite[Theorem 7]{EvansGariepy}). By Monotone Convergence Theorem we have
\begin{align}
\label{passaggi}
\int_{A} u(x) C_P(Df(x)) d \mathcal{S}^{Q}(x) &= \sum_{k=1}^{\infty}  \frac{1}{k} \int_{A \cap A_k}  C_P(Df(x)) d \mathcal{S}^{Q}(x) \notag \\
& \leq \sum_{k=1}^{\infty}  \frac{1}{k} L \int_{\M} \mathcal{S}^{Q-P}(f^{-1}(m) \cap A \cap A_k) d \mathcal{S}^P(m) \notag \\
& \leq \sum_{k=1}^{\infty}  \frac{1}{k} L \int_{\M} \int_{f^{-1}(m) \cap A} \bu_{A_k}(x) d\mathcal{S}^{Q-P}(x) d \mathcal{S}^P(m) \\
&=   L \int_{\M} \int_{f^{-1}(m) \cap A } \sum_{k=1}^{\infty}  \frac{1}{k} 1_{A_k}(x) d \mathcal{S}^{Q-P}(x) d \mathcal{S}^P(m) \notag \\
& = L \int_{\M} \int_{f^{-1}(m) \cap A} u(x) d\mathcal{S}^{Q-P}(x) \ d \mathcal{S}^P(m)\notag .
\end{align}
\end{proof}
\begin{coro}
\label{corfunzione}
In the hypotheses of Theorem \ref{theoremmeasurable},
let $u: A \to \R$ be a measurable function. If we assume that
\begin{itemize}
\item $u$ is $\mathcal{S}^{Q-P}$-summable on $f^{-1}(m) \cap A$ for $\mathcal{S}^{P}$-a.e. $m \in \M$ \item $\int_{\M} \int_{f^{-1}(m) \cap A} |u(x)| d \mathcal{S}^{Q-P}(x) d \mathcal{S}^{P}(m) < \infty$
\end{itemize}
then $u$ is summable on $A$.
\end{coro}
\begin{proof}
We can write $u=u^+-u^-$.
By proceeding analogously to (\ref{passaggi}), considering Theorem \ref{coareamagnani} instead of Theorem \ref{maintheorem2}, we then obtain that 
\begin{equation}
\label{seconda} - \int_{A} u^-(x) C_P(Df(x)) d \mathcal{S}^{Q}(x) \leq - \int_{\M} \int_{f^{-1}(m) \cap A} u^-(x) d\mathcal{S}^{Q-P}(x)  d \mathcal{S}^P(m).
\end{equation}
Hence by (\ref{passaggi}) applied to $u^+$, (\ref{seconda}), and our hypothesis, we have
\begin{align}
  \int_A u(x) C_P(Df(x)) d \mathcal{S}^Q(x) = & \int_A u^+(x) C_P(Df(x)) d \mathcal{S}^Q(x)\notag\\
  & - \int_{A} u^-(x) C_P(Df(x)) d \mathcal{S}^{Q}(x)\notag \\
\leq & L \int_{\M} \int_{f^{-1}(m) \cap A} u^+(x) d\mathcal{S}^{Q-P}(x)  d \mathcal{S}^P(m)\\
&- \int_{\M} \int_{f^{-1}(m) \cap A} u^-(x) d\mathcal{S}^{Q-P}(x)  d \mathcal{S}^P(m) \notag\\
 \leq & L \int_{\M} \int_{f^{-1}(m) \cap A} |u(x)| d\mathcal{S}^{Q-P}(x)  d \mathcal{S}^P(m)  < \infty. \notag
\end{align}
Now, it is enough to prove that $C_P(Df(x))>0$, for every $ x \in A$. This follows from the facts that $|Df(x)|>0$ for every $x \in A $ and that, taking into consideration Proposition \ref{normeconfronto}, for every $x$ we have the following
\begin{equation}
\label{coareafactormaggiorezero}
\begin{aligned}
&\frac{\mathcal{S}^{Q-P}(\ker(Df(x)) \cap B(0,1))}{\mathcal{H}_E^{q-p}(\ker(Df(x)) \cap B(0,1))} \\ =& \ \frac{1}{ \sup_{ w \in B(0,1)} \mathcal{H}_E^{q-p}(  B(w,1) \cap  \ker (Df(x))) } \\
\geq &  \ \frac{1}{ \mathcal{H}_E^{q-p}(  B(0,2) \cap  \ker (Df(x))) } \\
= &  \ \frac{1}{ \mathcal{L}^{q-p}(  B(0,2) \cap  \ker (Df(x))) } \\
\geq &  \ \frac{1}{ \mathcal{L}^{q-p}(  B_E(0, 2 C_{B(0,2)}) \cap  \ker (Df(x))) } \\
= & \frac{1}{(2C_{B(0,2)})^{q-p}}>0.
\end{aligned}
\end{equation}
\end{proof}
\begin{coro}
\label{corsommabile}
In the hypotheses of Theorem \ref{theoremmeasurable}, if $\bu_{A}(x)=0$ for $\mathcal{S}^{Q-P}$-a.e. $x \in f^{-1}(m)$, for $\mathcal{S}^{P}$-a.e. $m \in \M$, then $\bu_{A}(x)=0$ for $\mathcal{S}^Q$-a.e. $x \in \G$.
\end{coro}
\begin{proof}
It follows by Theorem \ref{theoremmeasurable} and (\ref{coareafactormaggiorezero}).
\end{proof}

\begin{rema} 
The hypothesis of Theorem \ref{maintheorem2} about the uniform Ahlfors regularity of the level sets of the map $f$ is not pointless: if we consider $f$ continuously Pansu differentiable with Pansu differential everywhere surjective on $\G$, the lower Ahlfors regularity of level sets is not always guaranteed, even locally. One can refer to \cite[Corollary 6.2.4]{Artem}, where explicit examples of this phenomenon are presented. It is still not known if pathological level sets are exceptional.
In addition, in \cite{Artem}, a class of mappings of higher regularity from the Heisenberg group $\H^n$ to the Euclidean space $\R^{2n}$ is studied. More precisely, the author considers functions $f \in C^{1, \alpha}_{\G}(\H^n,\R^{2n})$ with $\alpha>0$ i.e. given a homogeneous distance $d$ on $\H^n$, continuously Pansu differentiable maps such that, for every $a,b \in \H^n$
$$ d_{\mathcal{L}(\H^n, \R^{2n})}(Df(a), Df(b)) \lesssim d(a,b)^{\alpha}.$$
By \cite[Corollary 5.5.6]{Artem}, if we assume that the Pansu differential of $f$ is everywhere surjective, the level sets of $f$ are uniformly locally Ahlfors $2$-regular with respect to $\mathcal{S}^2$. Therefore, the validity of the inequality of Theorem \ref{maintheorem2} for this class of regular functions is ensured by our result. This confirms
the coarea-type equality proved in \cite[Theorem 6.2.5]{Artem}. 
To summarize, in this setting, we have weakened the hypothesis adopted in \cite[Theorem 6.2.5]{Artem} about the required regularity of the considered map, passing from $C^{1, \alpha}_{\G}$-regular maps, with $\alpha>0$, to continuously Pansu differentiable functions. We compensate the lower regularity of the map with the more geometrical hypothesis about the Ahlfors regularity of its level sets. We need to remark that these considerations are limited, up to now, to maps from the Heisenberg group $\H^n$ to $\R^{2n}$, about which more results are available.
\end{rema}

Let us now recall the definition of complementary subgroups of a Carnot group $\G$.
\begin{defi}
\label{complementary}
Let $\G$ be a Carnot group. Two homogeneous subgroups $\W$, $\V$ are said \textit{complementary subgroups of $\G$} if
 $\W \cap \V = \{0 \}$ and
for every $ g \in \G$, there exist $w \in \W$, $v \in \V$ such that:
$ g= w   v$, i.e. $ \G= \W   \V.$
\end{defi}
We denote by $\pi_{\W}: \G \to \W$ and $\pi_{\V}: \G \to \V$ the group projections on the subgroups: if $g=wv$ with $w \in \W$, $v \in \V$, $\pi_{\W}(g)=w$ and $\pi_{\V}(g)=v$.\\

Now we see how to apply Theorem \ref{maintheorem2} to the particular geometrical case in which there exists a $p$-dimensional homogeneous subgroup $\V$ complementary to $\ker(Df(x))$ for any point $x$ of a neighbourhood of a fixed compact set $\Omega$.
The proof of this observation uses tools of the theory of intrinsic Lipschitz graphs (\cite{IntLipgraphs}), so
we recall some definitions and results about splitting a Carnot groups into the product of complementary subgroups and about regularity of maps acting between homogeneous subgroups (for more details, please refer to \cite[Section 4]{Notesserra}).

%If $r>0$ and $w \in \W$, we denote by $B^{\W}(w,r):=B(w,r) \cap \W$ and if $v \in \V$, $B^{\V}(v,r):=B(v,r) \cap \V$.
\begin{prop}{\cite[Proposition 2.12]{IntLipgraphs}}
\label{propczero}
If $\G=\W\V$ is the product of two complementary subgroups, there exists $c_0=c_0(\W, \V)>0$ such that
\begin{equation}
\label{eqczero}
c_0 ( \| w \|+ \|v \|) \leq \|wv\| \leq \|w \|+ \|v \|
\end{equation}
for all $w \in \W$, $v \in \V$.
\end{prop}
\begin{defi}
\label{intgraph}
Let $\G=\W  \mathbb{V}$ be the product of two complementary subgroups and let $ U \subset\W$ be a set. If we consider a map $\phi: U \to \mathbb{V}$,
we define its \textit{intrinsic graph} as the set
$$ \mathrm{graph}(\phi)=\{w\phi(w):  w \in U \}.$$
The map $\Phi: U \to \mathrm{graph}(\phi), \ \Phi(w):= w  \phi(w)$ is called the \emph{graph map} of $\phi$.
\end{defi}

\begin{defi}
Let $\G=\W \V$ be the product of two complementary subgroups and let $L$ be a constant. Let $U \subset \W$ be open, we say that a function $\phi:U \to \V$ is \textit{intrinsic $L$-Lipschitz} if 
\begin{equation}
\label{defintlip}
 \| \pi_{\V}( \Phi(w')^{-1} \Phi(w) ) \| \leq L \| \pi_{\W}( \Phi(w')^{-1} \Phi(w) ) \|
\end{equation}
for every $w ,w' \in U$.
\end{defi}

Intrinsic Lipschitz graphs are lower Ahlfors regular.
\begin{prop}{\cite[Theorem 3.9]{IntLipgraphs}}
\label{intLipAhlf}
Let $\G=\W\V$ be the product of two complementary subgroups endowed with a homogeneous distance $d$. Let $N$ be the metric dimension of $\W$. If $\phi: \W \to \V$ is an intrinsic $L$-Lipschitz map on $\W$, then 
$$\left(\frac{c_0(\W,\V)}{1+L} \right)^{N} r^{N} \leq \mathcal{S}^{N}(\mathrm{graph}(\phi) \cap B(x,r))$$
for all $x \in \mathrm{graph}(\phi)$ and $r>0$.
\end{prop}

\begin{rema}
The proof of Proposition \ref{intLipAhlf} is based on local arguments. Let us now assume that $\phi$ is defined from an open set $A \subset \W $ to $\V$ and set $w \in A$. Assume that, for a positive constant $\tilde{r}>0$, $\pi_{\W}(B(\Phi(w), \tilde{r}))\subset  A$, then for any $0<r < \tilde{r}$ we get that $\mathcal{S}^{N}(\mathrm{graph}(\phi) \cap B(\Phi(w),r)) \geq \left(\frac{c_0(\W,\V)}{1+L} \right)^{N} r^{N}.$ 
\end{rema}

Let us introduce a particular class of h-homomorphism between two Carnot groups, for more details please refer to \cite[Definition 2.5, Proposition 7.10]{Magnani2008}.
\begin{defi}
Given two Carnot groups $\G$ and $\M$. If a map $L: \G \to \M$ is a h-homomorphism and $\K$ is its kernel, we call $L$ a \textit{h-epimorphism} if $L$ is surjective and there exists a homogeneous subgroup of $\G$, $\H $, complementary to $\K$. In this case the restriction $L |_{\H}$ is a h-isomorphism.
\end{defi}

Now we report a result contained in \cite[Lemma 2.10]{AntoineSebastiano}, that is, up to now, the most general available implicit function theorem in this setting. Similar statements are proved in \cite[Theorem 1.4]{Magnani_2013} and \cite[Theorem 3.27]{Areaformula}.
\begin{theo}
\label{IFT}
Let $\G$ and $\M$ be two Carnot groups and let $\Omega \subset \G$ be open. Let $f \in C^1_{\G}(\Omega, \M)$ be a function and fix $x_0 \in \Omega.$ Assume that $Df(x_0)$ is a h-epimorphism and consider $\V$ a subgroup complementary to $\ker(Df(x_0))$. Fix a homogeneous subgroup $\W$ complementary to $\V$. Write $x_0=w_0 v_0$ with respect to the splitting $\W \V$.

Then there exist an open set $A \subset \W$, with $w_0 \in A$, and a continuous map $\phi: A \to \V$ such that $f(a \phi(a))=f(x_0)$ for every $a \in A$.
\end{theo}

\begin{rema}
By \cite[Corollary 2.16]{AntoineSebastiano}, the parametrization $\phi$ given by Theorem \ref{IFT} is $L$-Lipschitz for some positive constant $L$.
\end{rema}

Let us fix again $(\G,d_1)$, $(\M, d_2)$ two Carnot groups, endowed with homogeneous distances, of metric dimension $Q, \ P$ and topological dimension $q, \ p$, respectively. For any set $\Omega \subset \G$, and any real number $D>0$ we set $\Omega_D=\{ y \in \G : d(y,\Omega) < D \}.$

By modifying the proof of \cite[Lemma 2.9]{AntoineSebastiano}, combining it with an easy compactness argument and Theorem \ref{uniformPdiff}, the following immediately follows.
\begin{prop}
\label{propdelC}
Let us consider a map $f \in C^1_{\G}(\G, \M)$ and a compact set $\Omega \subset \G$. Assume that there exists a $p$-dimensional homogeneous subgroup $\V$ such that $Df(x)|_{\V}:\V \to \M$ is a h-isomorphism for every $x \in \Omega$. Then there exists a constant $R>0$ such that for every $x \in \Omega$, for every $y \in B(x,R)$ and $v \in \V$ such that $yv \in B(x,R)$
$$ d_2(f(y),f(yv)) \geq R \| v \|_1.$$
\end{prop}

Notice that our hypothesis implies that $\V$ is complementary to $\ker(Df(x))$ for every $x \in \Omega$.

Indeed, any continuously Pansu differentiable map is locally metric Lipschitz. Hence, combining Proposition \ref{propdelC} with the proof of \cite[Corollary 2.16]{AntoineSebastiano} we get the following.

\begin{prop}
\label{propuniflip}
Let us consider a map $f \in C^1_{\G}(\G, \M)$ and a compact set $\Omega \subset \G$. Let us assume that there exists a $p$-dimensional homogeneous subgroup $\V$ such that $Df(x)|_{\V}:\V \to \M$ is a h-isomorphism for every $x \in \overline{\Omega_D}$ for some $D>0$. Then there exists a constant $L$ such that for every $m \in \M$ and $x \in f^{-1}(m) \cap \Omega $, the set $f^{-1}(m) \cap B(x,R)$ is an intrinsic Lipschitz graph with constant $L$, where $R$ is the constant of Proposition \ref{propdelC} applied to $\Omega$.
\end{prop}

\begin{proof}
By hypothesis, at any point $x \in 	\overline{\Omega_D}$, $\ker Df(x)$ is a normal homogeneous subgroup complementary to $\V$, hence $Df(x)$ is a h-epimorphism.
Assume that $R$ is smaller than $D$.
Let us fix a homogeneous subgroup $\W$ complementary to $\V$. For every $m \in \M$ and $x \in f^{-1}(m) \cap \Omega$, the set $f^{-1}(m) \cap B(x,R)$ is contained in the intrinsic graph of a function $\phi_{m,x} : U_{m,x} \subset \W \to \V$, for some open set $U_{m,x} \subset \W$. The map $\phi_{m,x}$ is given by Theorem \ref{IFT}, repeatedly applied to different points of $f^{-1}(m) \cap B(x,R)$, if necessary. 

Now we need to observe that the notion of intrinsic Lipschitz function introduced in \cite{AntoineSebastiano} is equivalent to our notion (it is immediate to compare Definition in \cite{AntoineSebastiano} with \cite[Definition 9, Definition 10, Proposition 3.1]{IntLipgraphs}).
Then by \cite[Corollary 2.16]{AntoineSebastiano}, $f^{-1}(m) \cap B(x,R)$ is the intrinsic Lipschitz graph of an intrinsic $L$-Lipschitz function $\phi_{m,x}$ for some constant $L$ depending on $R$, and on the Lipschitz constant of $f|_{B(x,R)}$, that can be uniformly bounded by the $\sup_{x \in \Omega} \mathrm{Lip}(f|_{B(x,R)})\leq \mathrm{Lip}(f|_{\overline{\Omega_D}})< \infty$. As a consequence, the sets $f^{-1}(m) \cap B(x,R)$ are intrinsic $L$-Lipschitz for some positive $L$, that can be chosen independent of $x\in \Omega$ and $m \in \M$.
\end{proof}

\begin{coro}
\label{conspezz}
Let $f \in C^1_{\G}(\G, \M)$ be a function with $Df(x)$ surjective at every $x \in \G$ and let $\Omega\subset \G$ be a compact set. Assume that there exists a $p$-dimensional subgroup $\V$ of $\G$ such that $Df(x)|_{\V}$ is an h-isomorphism for every $x \in \overline{\Omega_D}$ for some $D>0$. Set $\lambda=\sup_{x \in \Omega} \mathrm{Lip}(f|_{B(x,R)})$, where $R$ is the constant given by Proposition \ref{propdelC} applied to $\Omega$. Then there exists a constant $1 \leq T(\G, \lambda, R, p) < \infty$, such that
$$ \int_{\Omega} C_P(Df(x)) d \mathcal{S}^{Q}(x) \leq T \int_{\M} \mathcal{S}^{Q-P}(f^{-1}(m) \cap \Omega)  d \mathcal{S}^P(m).$$
\end{coro}

\begin{proof}
We can assume that $R<D$. Set $\W$ any homogeneous subgroup complementary to $\V$.
By Propositions \ref{propuniflip} and \ref{intLipAhlf}, there exists a constant $K>0$ such that for every $m \in \M$ and $x \in f^{-1}(m) \cap \Omega$, for every $0<r<R$, $\mathcal{S}^{Q-P}(f^{-1}(m) \cap B(x,r)) \geq Kr^{Q-P} $, where $K$ is a constant depending on $c_0(\W,\V)>0$ and on the intrinsic Lipschitz constants of the parametrizing maps $\phi_{m,x}:U_{m,x} \subset \W \to \V$ of $\{f^{-1}(m) \cap B(x,r)\}_{\{m \in \M, x \in f^{-1}(m) \cap \Omega\}}$. Moreover observe that by Proposition \ref{propuniflip}, $\phi_{m,x}$ are intrinsic $L$-Lipschitz, for some constant $L$ independent of $m$ and $x$.
Now notice that this observation can take the place of the hypothesis that level sets $f^{-1}(m)$ are uniformly locally lower Ahlfors $(Q-P)$-regular with respect to $\mathcal{S}^{Q-P}$ in Theorem \ref{maintheorem2} (more precisely in Claim 2 on Theorem \ref{maintheorem}), hence we can apply our result to this situation, and we directly get the thesis.
\end{proof}

\begin{rema}
We have seen, in the proof of Corollary \ref{conspezz}, that the existence of a $p$-dimensional homogeneous subgroup $\V$ complementary to $\ker(Df(x))$ for every point $x \in \G$, implies that the level sets of $f$ are $R$-locally $C$-lower Ahlfors $(Q-P)$-regular with respect to $\mathcal{S}^{Q-P}$, for some positive constants $C$ and $R$, locally independent of the choice of the level set. We want to highlight that the opposite may be false. In fact, there exist continuously Pansu differentiable maps between Carnot groups, with everywhere surjective differential, such that their level sets are lower Ahlfors regular, but at the same time $\ker(Df(x))$ does not admit any complementary subgroup. 

We present a simple example related to the first Heisenberg group $\H^1$, that is the simplest non-commutative Carnot group. It can be represented as a direct sum of two linear subspaces $\H^1=H_1 \oplus H_2$, where $H_1=\mathrm{span}(e_1, e_2)$, $H_2=\mathrm{span}(e_3)$ with unique non trivial relation $[e_1, e_2]=e_3$. For every $p,q \in \H^1$, $pq=p+q+\frac{1}{2}[p,q]$.

Let us consider the map
$$f:\H^1 \to \R^2, \ f(x,y,z)=(ax+by,cx+dy), \mathrm{ \ with \ } \det \begin{bmatrix}
a & b\\
c & d \\
\end{bmatrix} \neq 0.$$
Observe that  $f \in \mathcal{L}(\H^1, \R^2)$, then the Pansu differential of $f$ is constant on $\H^1$: for every $\bar{x} \in \H^1$,
$$Df(\bar{x})(x,y,z)=f(x,y,z),$$
hence, $\ker(Df(\bar{x}))= \mathrm{span}(e_3) $ for every $\bar{x} \in \H^1$.
Notice that $\mathrm{span}(e_3)$ is a normal homogeneous subgroup of metric dimension 2 that does not admit any complementary subgroup (see for instance \cite[Proposition 4.1]{Diffofintr}).
Let us now focus on the level sets of $f$. If we fix $v \in \R^2$, $f^{-1}(v)=w   \mathrm{span}(e_3)$ for some $w=w(v) \in H_1$, hence any level set is a coset of $\mathrm{span}(e_3)$. Then, by left invariance and homogeneity of the distance, the level sets $f^{-1}(v)$ are $C$-lower Ahlfors 2-regular with respect to $\mathcal{S}^2$, for some positive constant $C$, independent of the choice of $v$.
\end{rema}

\textbf{Acknowledgements}: We would like to thank Pierre Pansu for directing us to the study of the problem during our stay in Orsay. We would also like to express our gratitude to Bruno Franchi and Francesco Serra Cassano, for many fruitful conversations and suggestions.

\end{document}